\documentclass[12pt]{article}
\usepackage{amssymb} 
\usepackage[T1]{fontenc} 
\usepackage{lmodern} 
\usepackage{graphicx} 
\usepackage{relsize} 
\usepackage{enumitem} 
\usepackage{amsmath} 
\usepackage{amsthm} 
\usepackage{float}
\usepackage{url}
\usepackage{cite}
\allowdisplaybreaks
\usepackage{array}

\DeclareMathOperator{\Res}{Res}
\DeclareMathOperator{\Syl}{Syl}
\theoremstyle{plain}
\newtheorem{thm}{Theorem}[section]
\newtheorem{cor}[thm]{Corollary}
\newtheorem{lem}[thm]{Lemma}

\theoremstyle{definition}

\author{Joanna Turaj}
\title{On some resultants formulas of Schur type}
\date{}

\begin{document}
\maketitle
\begin{abstract}
Let \((r_{A,n}(x))_{n \in \mathbb{N}}\) be a sequence of polynomials with coefficients from a field \(K\) satisfying the recurrence relation 
\[r_{A,n}(x)= \sum_{|\alpha|\leq m} t_{\alpha,n}(x)\textbf{r}_{A,n}^\alpha(x)\]
of order \(d+1 \in \mathbb{N}_{+}\), where \(t_{\alpha,n} \in K[x]\), \(m \in \mathbb{N}_{+}\) are fixed, \(\alpha \in \mathbb{N}^{d+1}\), \(|\alpha| = \alpha_0 + \ldots+\alpha_d\) and
\[\textbf{r}_{A,n}^\alpha(x)=r_{A,n-1}^{\alpha_0}(x)r_{A,n-2}^{\alpha_1}(x)\cdots r_{A,n-d-1}^{\alpha_d}(x).\]
We show that under mild assumptions on the initial polynomials \(r_{A,0}, \ldots, r_{A,d}\) and the coefficients \(t_{\alpha,n}\), we can give the expression for the resultant \( \Res(r_{A,n}, r_{A,n-1})\). Our results generalize recent result of Ulas concerning the case \(m=1\) and \(d=1\).

\end{abstract}

\section{Introduction}
Let \(K\) be a field and let \(\mathbb{N}\) denote the set of non-negative integers, \(\mathbb{N}_{+}\) the set of positive integers, and for given \(k \in \mathbb{N}\) define \( \mathbb{N}_{\geq k }\) as the set of integers greater or equal to \(k\). 

The resultant \(\Res(F,G)\) of two polynomials \(F, G \in K[x]\) first appeared in the work of J. J. Sylvester. As an example, he considered equations of the form 
\begin{align}\label{equations}
\begin{split}
ax^2+bx+c&=0,\\
lx^2+mx+n&=0,
\end{split}
\end{align}
and the corresponding matrix
\[A=\begin{bmatrix}
a & b& c&0 \\
0 &a & b& c\\
l & m&n &0\\
0 &l &m&n
\end{bmatrix}.\]
The determinant \(\det(A)\) (which is equal to the resultant of the polynomials on the left-hand side of the equations (\ref{equations})) vanishes if and only if the chosen polynomials have a common factor \cite{sylvester1840xxiii}.  More generally, for arbitrary polynomials \(F,G \in K[x]\) their resultant \(\Res(F,G)=0\) if and only if they have a common root in a fixed algebraic closure of \(K\) (definition will be presented in Section \ref{sect2}). Resultants can be used for solving systems of algebraic equations by reducing the problem to the search for roots of polynomials in one variable \cite[Section 1.3]{prasolov2004polynomials}.

We can ask for an expression for the resultants of polynomials given by a recurrence. The classical result in this field is Schur's formula \cite[Chapter VI]{szeg1939orthogonal}. Strictly speaking, let \((r_n(x))_{n\in\mathbb{N}}\) be a sequence of polynomials defined by the recurrence relation
\[r_n(x)=(a_nx+b_n)r_{n-1}(x)-c_nr_{n-2}(x), \; n \geq 2,\]
where \(r_0(x)=1\), \(r_1(x)=a_1x+b_1\), and \(a_n,b_n,c_n \in \mathbb{C}\) such that \(a_nc_n\neq 0\). The resultant of the polynomials \(r_{n}, r_{n-1}\) is given by the formula
\[\Res(r_n, r_{n-1})=(-1)^{n(n-1)/2}\prod_{i=1}^{n-1}a_i^{2(n-i)} c_{i+1}^i.\]
The original result of Schur can be found in \cite{schur1931affektlose} (in German).

Recently this result was generalized by Ulas \cite{ulas2021generalization}. He considered the case when the initial polynomials \(r_0, r_1\) of the recurrence are not necessarily of degrees \(0\) and \(1\), respectively, and the coefficients of this recurrence can be more complicated polynomials. More precisely, for an element \(B\) of the set \( \mathcal{B}= \{(i,j,k,m)\in \mathbb{N}^{4}: i \leq j, m \leq k\}\) he defined a sequence of polynomials \((r_{B,n}(x))_{n\in \mathbb{N}}\) by the recurrence
\[r_{B,n}(x)= f_n(x)r_{B,n-1}(x) -v_nx^mr_{B,n-2}(x), \; n \geq 2,\]
where \(v_n \in K\) and \(r_{B,0}, r_{B,1}, f_n \in K[x]\) are of degrees \(i,j,k\), respectively. Under those assumptions, he gave a formula for \(\Res(r_{B,n}, r_{B,n-1})\).

It might be asked whether the results of Ulas can be further generalized, especially for a polynomial sequence given by a more general (especially non-linear) recurrences. In this paper we extend the theorem presented in \cite{ulas2021generalization} to the situation where the sequence \((r_{A,n}(x))_{n \in \mathbb{N}}\) is described by quite general, non-linear recurrence relation of order \(d+1 \in \mathbb{N}_+\). More precisely, we assume that the polynomials in the sequence \((r_{A,n}(x))_{n\in\mathbb{N}}\) satisfy the following recurrence relation 
\begin{equation}\label{recurrence}
r_{A,n}(x)=g_{n}(x)r_{A,n-1}^m(x) +\sum_{|\alpha|<m} t_{\alpha,n}(x)\textbf{r}_{A,n}^\alpha(x)r_{A,n-1}(x)+v_{n}x^lr_{A,n-2}^m(x),
\end{equation}
where \(n \geq d+1\), \(v_n \in K\), \(g_{n}, t_{\alpha,n} \in K[x]\), \(\alpha \in \mathbb{N}^{d+1}\), \(|\alpha|=\alpha_0 +\ldots +\alpha_d\), and 
\[\textbf{r}_{A,n}^\alpha(x)= r_{A,n-1}^{\alpha_0}(x)r_{A,n-2}^{\alpha_1}(x)\cdots r_{A,n-d-1}^{\alpha_d}(x).\]
We provide some assumptions which allow us to provide the expression for the resultant \(\Res(r_{A,n},r_{A,n-1})\), for \(n\geq d+1\). According to our best knowledge, our result is the first one which allows to compute the resultants for a broad class of polynomial sequences given by a non-linear recurrence relation of arbitrary (but fixed) order.

Let us describe the content of the paper in some details. In Section \ref{sect2}, we discuss the basic properties satisfied by the resultants. Section \ref{sect3} is dedicated to the main result of the paper. We give the sufficient assumptions under which we will be able to obtain the formula for the resultant \(R_n=\Res(r_{A,n},r_{A,n-1})\), where the  sequence \( (r_{A,n}(x))_{n\in \mathbb{N}}\) is given by the recurrence (\ref{recurrence}). Lastly, we present the cases for which our theorem is equivalent to Schur's formula and the work presented by Ulas in \cite{ulas2021generalization}.

\section{Resultants and their basic properties} \label{sect2}

Let \(F, G\) be polynomials in one variable over a field \(K\) such that
\begin{align*}
F(x)&=a_nx^n +a_{n-1}x^{n-1}+ \ldots + a_1 x +a_0,\\
G(x) &= b_m x^m +b_{m-1}x^{m-1} +\ldots +b_1x+b_0,
\end{align*}
where \(a_nb_m \neq 0\). The resultant of the polynomials \(F, G\) is defined as  
\[\Res(F,G)=a_n^mb_m^n\prod_{i=1}^n \prod_{j=1}^m (\alpha_i-\beta_j),\]
where \(\alpha_1, \ldots,\alpha_n\) and \(\beta_1, \ldots, \beta_m\) are the roots of \(F\) and \(G\) in a fixed algebraic closure of \(K\), respectively. Equivalently, we can define the resultant of \(F, G\) as the determinant of the Sylvester matrix of dimension \((n+m)\times (n+m)\)
\[\Syl(F,G)=\begin{bmatrix}
a_n &a_{n-1}& &\cdots & a_1 & a_0 & & &\\
& a_n &a_{n-1}& & \cdots &a_1 & a_0 & &\\
&&\ddots &\ddots& &&\ddots&\ddots &\\
&&&a_n&a_{n-1}&&\cdots &a_1&a_0\\
b_m & b_{m-1}& &\cdots &b_1& b_0 & & &\\
& b_m & b_{m-1}& & \cdots &b_1 & b_0 & &\\
&&\ddots & \ddots& &&\ddots&\ddots &\\
&&&b_m& b_{m-1}&&\cdots &b_1&b_0\\
\end{bmatrix}.\]
More precisely, if \(\Syl(F,G)= [c_{i,j}]_{1\leq i,j\leq n+m}\), then 
\begin{align*}
c_{i,j}=a_{n-j+i}, \quad \text{for} \quad 1\leq i \leq m, \\
c_{m+i,j}=b_{m-j+i}, \quad \text{for} \quad 1\leq i \leq n,
\end{align*}
where \(a_i = 0\) for \(i \notin \{0,\ldots,n\}\) and \(b_j = 0\) for \(j \notin \{0,\ldots,m\}\). 

We now recall some properties of resultants that will be necessary later in this paper. For polynomials \(F,G\) defined above and \(H \in K[x]\), we have that
\begin{align}
\Res(F,0)&=0,\notag \\
\Res(F,G)&=(-1)^{nm}\Res(G,F),\label{symmetry}\\
\Res(F,GH)&=\Res(F,G)\Res(F,H),\label{multiplicativity}\\
\Res(F,G)&= a_n^m \prod_{i=1}^n G(\alpha_i)= (-1)^{nm}b_m^n\prod_{j=1}^m F(\beta_j),\label{zeros}
\end{align}
and if \(F(x)=a_0\) is a nonzero constant polynomial, then
\begin{equation}\label{constant}
\Res(F,G)=\Res(a_0,G)=\Res(G,a_0)=a_0^m.
\end{equation}
The proofs of the above properties can be found in \cite[Section 3.6]{mignotte2012mathematics}. Finally, we present the last property for polynomials \(F\) and \(G\) satysfying the equality \(F(x)=q(x)G(x)+r(x)\),  which will be crucial in the proof of our main result (Theorem \ref{main}).
\begin{lem}\label{lemma}
Let \(F,G \in K[x]\) be polynomials of the following form:
\begin{align*}
F(x)&=a_nx^n +a_{n-1}x^{n-1}+ \ldots + a_1 x +a_0,\\
G(x) &= b_m x^m +b_{m-1}x^{m-1} +\ldots +b_1x+b_0,
\end{align*}
where \(\deg F=n\geq m=\deg G\), and suppose that \(F(x)=q(x)G(x)+r(x)\) for certain \(q,r \in K[x]\), \(\deg r=k<m\). Then the resultant of \(G,F\) is given by the formula
\begin{equation}\label{division}
\Res(G,F)= b_m^{n-k}\Res(G,r).
\end{equation}
\end{lem}
More on the properties of resultants can be found in \cite{gelfand2008discriminants,mignotte2012mathematics,pohst1997algorithmic,prasolov2004polynomials}.

\section{Main result}\label{sect3}
The goal of this section is to provide a generalization of Ulas' theorem \cite{ulas2021generalization}. To state our results we first need to introduce some notation. We define the set 
\[\mathcal{A}=\{(i_0, i_1,\ldots, i_{d-1}, i_{d},k,l,m)\in \mathbb{N}^{d+4}: i_d\geq i_{d-1} \geq \ldots \geq i_0, \; k\geq l, m\neq 0\}.\]
For \(A \in \mathcal{A}\) we consider the polynomials
\[ r_{A,0}(x)=\sum_{s=0}^{i_0} p_{s,0}x^s, \, r_{A,1}(x)= \sum_{s=0}^{i_1} p_{s,1}x^s, \ldots, \, r_{A,d}(x)=\sum_{s=0}^{i_d} p_{s,d}x^s,\]
and the recurrence
\[r_{A,n}(x)=g_{n}(x)r_{A,n-1}^m(x) +\sum_{|\alpha|<m} t_{\alpha,n}(x)\textbf{r}_{A,n}^\alpha(x)r_{A,n-1}(x)+v_{n}x^lr_{A,n-2}^m(x),\]
where \(\alpha = (\alpha_0,\alpha_1, \ldots, \alpha_d) \in \mathbb{N}^{d+1}\), \(|\alpha|=\alpha_0+\ldots+\alpha_d\), and
\begin{align*}
\textbf{r}_{A,n}^\alpha(x)&=r_{A,n-1}^{\alpha_0}(x)r_{A,n-2}^{\alpha_1}(x)\cdots r_{A,n-d-1}^{\alpha_d}(x),\\
g_{n}(x) &= \sum_{s=0}^k a_{s,n} x^s,
\end{align*}
and \(v_n \in K\). We assume that
\begin{align}
t_{\alpha,n}(0)&=0,\notag \\
\deg t_{\alpha,n}&<\deg g_{n}, \notag\\
a_{k,n}\prod_{s=0}^d p_{i_s,s}&\neq 0 \quad \text{ for } n \in \mathbb{N}.\label{eq1}
\end{align} 
Moreover, if \(i_d=i_{d-1}\) and \(k=l\), then
\begin{equation}\label{eq2}
a_{k,d+1}p_{i_d,d}^m +v_{d+1}p_{i_{d-1},d-1}^m \neq 0.
\end{equation}
The above assumptions allow us to give formulas for the degree of the polynomial \(r_{A,n}\), its leading term, the value \(r_{A,n}(0)\), and lastly, the resultant \(\Res(r_{n},r_{n-1})\). Those expressions depend only on the choice of the element \(A \in \mathcal{A}\), the leading and constant terms of the polynomials \(r_{d}\), \(r_{d-1}\), \(g_n\) and the coefficients \(v_n\) (that are all fixed). We prove the following:

\begin{thm}\label{main}
Under the above conditions, the resultant \(R_n\) of the polynomials \(r_{A,n}, r_{A,n-1}\) is given by the formula
\[R_n = (-1)^{\sum_{s=d+1}^{n}m^{n-s}e_A(s)}R_d^{m^{n-d}}\prod_{s=d+1}^{n} \left(L_{s-1}^{\gamma_A(s)}v_s^{\deg r_{s-1}}C_{s-1}^l \right)^{m^{n-s}},\] 
where
\begin{align*} 
e_A(n)&=\deg(r_{A,n})\cdot \deg(r_{A,n-1}),\\
\gamma_A(n)&=\deg(r_{A,n})-\deg(v_nx^lr_{A,n-2})\\
&=\begin{cases}
k-l+m(i_d-i_{d-1}) &\text{ for } n=d+1,\\
m^{n-d-1}(k+i_d(m-1))+k-l &\text{ for } n\geq d+2,
\end{cases}\\
C_n&= \begin{cases}
1 &\text{ for } l=0,\\
\displaystyle p_{0,d}^{m^{n-d}}\prod_{s=1}^{n-d-1}a_{0,d+s}^{m^{n-d-s}} &\text{ for } l>0,
\end{cases}
\end{align*}
and \[L_n=\begin{cases}
\displaystyle \left(a_{k,d+1}p_{i_d,d}^m +v_{d+1}p_{i_{d-1},d-1}^m\right)^{m^{n-d-1}} \prod_{s=2}^{n-d} a_{k,d+s}^{m^{n-d-s}} &\text{for } i_d=i_{d-1}, k=l,\\
\displaystyle p_{i_d,d}^{m^{n-d}} \prod_{s=1}^{n-d} a_{k,d+s}^{m^{n-d-s}} &\text{otherwise}.
\end{cases}\] 
Moreover, the degree of \(r_{A,n}\) is equal to
\[\deg(r_{A,n}) = k \sum_{s=0}^{n-d-1}m^s +i_dm^{n-d} \, \text{ for } n \geq d+1.\]
\end{thm}
\begin{proof}
Because \(A\in \mathcal{A}\) is fixed, to simplify the notation we omit the index \(A\) in \(r_{A,n}\) and \(\textbf{r}_{A,n}^\alpha\), i.e., we will write \(r_{n}\) instead of \(r_{A,n}\) and \(\textbf{r}_{n}^\alpha\) in place of \(\textbf{r}_{A,n}^\alpha\). Under the assumptions on \(i_0, \ldots, i_d ,k,l\), the degrees of the polynomials \(t_{\alpha,n}\), and the conditions (\ref{eq1}), (\ref{eq2}), we can calculate the degree of the polynomial \(r_n\), its leading term \(L_n\), and the value of \(C_n \mathrel{\mathop:}= r_{n}(0)\). Let us consider the first step of the recurrence:
\begin{align*}
r_{d+1}(x)&=g_{d+1}(x)r_{d}^m(x)+\sum_{|\alpha|<m}t_{d+1}(x)\textbf{r}_{d+1}^\alpha(x)r_{d}(x)+v_{d+1}x^l r_{d-1}^m(x)\\
&=r_d(x)\left(g_{d+1}(x)r_{d}^{m-1}(x) + \sum_{|\alpha|<m}t_{\alpha,d}(x)\textbf{r}_{d+1}^\alpha(x)\right) + v_{d+1}x^l r_{d-1}^m(x).
\end{align*}
If \(i_d >i_{d-1}\) and \(k \geq l\), the term \(v_{d+1}x^l r_{d-1}^m(x)\) does not affect \(\deg(r_{d+1})\). It is also true in the case when \( i_d=i_{d-1}\) and \(k>l\). Therefore,
\[\deg(r_{d+1}) = i_d + \deg \left(g_{d+1}r_{d}^{m-1}+ \sum_{|\alpha|<m} t_{\alpha,d+1}\textbf{r}_{d+1}^\alpha \right).\]
Due to our assumptions, for each \(\alpha\) satysfying \(|\alpha|<m\), the degree of the polynomial \( g_{d+1}r_{d}^{m-1}\) is bigger than the degrees of the terms \(t_{\alpha, d+1} \textbf{r}_{d+1}^\alpha\). More precisely, we have the equality \(\deg (g_{d+1}r_{d}^{m-1})= k+(m-1)i_d\). Hence,
\[\deg(r_{d+1})=k+i_d m.\]

In the case where \(i_d= i_{d-1}\) and \(k=l\), by the assumption (\ref{eq2}) we obtain the same result. Analogously, we can write
\[r_{d+2}(x)=r_{d+1}(x)\left( g_{d+2}(x)r_{d+1}^{m-1}+\sum_{|\alpha|<m} t_{\alpha, d+2}(x) \textbf{r}_{d+2}^\alpha(x)\right) +v_{d+2}x^l r_{d}^m(x),\]
and the degree of \(g_{d+2} r_{d+1}^m\) is greater than the degrees of the other terms in the sum. This implies that
\[\deg(r_{d+2}) = k+ m(k+i_dm)= mk +k +i_dm^2,\]
and by induction on \(n\) we get
\[\deg(r_{d+n})=k\sum_{s=0}^{n-1} m^{s} + i_d m^{n} \;\text{ for } n\geq 1. \]
Equivalently, we can write
\[\deg(r_{n}) =k\sum_{s=0}^{n-d-1} m^{s} +i_d m^{n-d}\; \text{ for } n \geq d+1.\] 

To compute the leading term \(L_n\) of \(r_n\), as before, we consider two cases. The first one is \( (i_d>i_{d-1})\vee (i_d=i_{d-1} \wedge k>l) \). If \(i_d > i_{d-1}\), then using previous calculations, we know that the degree of the polynomial \(r_{d+1}\) is determined by the term \(g_{d+1}r_{d}^m\). Therefore, the leading term \(L_{d+1}\) is equal to \( a_{k,d+1}p_{i_d,d}^m\). Similarly, \(L_{d+2} = a_{k,d+2}\left(a_{k,d+1}p_{i_d,d}^m \right)^m\), and by an easy induction
\[L_{d+n}= p_{i_d,d}^{m^{n}} \prod_{s=1}^{n} a_{k,d+s}^{m^{n-s}} \;\text{ for } n \geq 1.\]
Equivalently,
\[L_{n} = p_{i_d,d}^{m^{n-d}} \prod_{s=1}^{n-d} a_{k,d+s}^{m^{n-d-s}} \;\text{ for } n \geq d+1.\]
For \(i_d=i_{d-1}\) and \(k=l\) we apply the same reasoning, and we get that
\begin{align*}
L_{d+1} &= a_{k,d+1}p_{i_d,d}^m +v_{d+1}p_{i_{d-1},d-1}^m,\\
L_{d+2} &= a_{k,d+2}\left( a_{k,d+1}p_{i_d,d}^m +v_{d+1}p_{i_{d-1},d-1}^m\right)^m,
\end{align*}
and in general 
\[L_{d+n} = \left(a_{k,d+1}p_{i_d,d}^m +v_{d+1}p_{i_{d-1},d-1}^m\right)^{m^{n-1}} \prod_{s=2}^{n} a_{k,d+s}^{m^{n-s}} \;\text{ for } n \geq 1.\]
Therefore,
\[L_{n} = \left(a_{k,d+1}p_{i_d,d}^m +v_{d+1}p_{i_{d-1},d-1}^m\right)^{m^{n-d-1}} \prod_{s=2}^{n-d} a_{k,d+s}^{m^{n-d-s}} \;\text{ for } n \geq d+1.\]

In order to find the value of the constant term \(C_n= r_n(0)\) we will consider two cases: \(l>0\) and \(l=0\). If \(l>0\), then from the assumption that \(t_{\alpha,n}(0)=0\) for all \( n \geq d+1\) and induction on \(n\) we get
\[C_{d+n}=p_{0,d}^{m^n}\prod_{s=1}^{n-1}a_{0,d+s}^{m^{n-s}} \;\text{ for } n \geq 1.\]
This equality can be equivalently written as
\[C_{n}=p_{0,d}^{m^{n-d}}\prod_{s=1}^{n-d-1}a_{0,d+s}^{m^{n-d-s}} \;\text{ for } n \geq d+1.\]
When \(l=0\), the values of \(C_n\) satisfy the recurrence:
\[C_n=a_{0,n}C_{n-1}^m +v_{n}C_{n-2}^m\]
for \(n\geq d+2\), where \(C_d=p_{0,d}\) and \(C_{d+1}= a_{0,d+1}p_{0,d}^m+v_{d+1}p_{0,d-1}^m\). We are not able to show the exact form of the term \(C_n\) but it is not necessary to prove the theorem.

Now, let \( R_n\) denote the resultant of the polynomials \(r_n\), \(r_{n-1}\), i.e., \(R_n = \Res(r_n, r_{n-1})\). For \(n=d+1\) we have the following chain of equalities:
\begin{align*}
R_{d+1}&\stackrel{(\ref{symmetry})}{=}(-1)^{\deg(r_{d+1})\deg(r_{d})}\Res(r_{d},r_{d+1})\\
&=(-1)^{i_d(k+i_dm)}\Res\left(r_{d}, r_{d}\left(g_{d+1}r_{d}^{m-1}+\sum_{|\alpha|<m}t_{\alpha,d}\textbf{r}_{d+1}^\alpha\right)+v_{d+1}x^l r_{d-1}^m\right)\\
&\stackrel{(\ref{division})}{=}(-1)^{i_d(k+i_dm)}p_{i_d,d}^{\deg(r_{d+1})-\deg(v_{d+1}x^l r_{d-1}^m)}\Res(r_{d},v_{d+1}x^l r_{d-1}^m)\\
&=(-1)^{i_d(k+i_dm)}p_{i_d,d}^{k+i_d m-(l+i_{d-1} m)}\Res(r_{d},v_{d+1}x^l r_{d-1}^m)\\
&\stackrel{(\ref{multiplicativity})}{=}(-1)^{i_d(k+i_dm)}p_{i_d,d}^{k-l+m(i_d-i_{d-1})}\Res(r_{d},v_{d+1})\Res(r_{d},x^l)\Res(r_{d},r_{d-1}^m)\\
&\kern-0.6em\stackrel{(\ref{constant}), (\ref{zeros})}{=}(-1)^{i_d(k+i_dm)}p_{i_d,d}^{k-l+m(i_d-i_{d-1})} v_{d+1}^{\deg(r_{d})}C_d^{l}\Res(r_{d},r_{d-1})^m\\
&=(-1)^{i_d(k+i_dm)}p_{i_d,d}^{k-l+m(i_d-i_{d-1})} v_{d+1}^{i_d}p_{0,d}^l\Res(r_{d},r_{d-1})^m,
\end{align*}
where the numbers above the equals signs correspond to the numbers of the resultant properties in Section \ref{sect2}. In general, the form of the considered recurrence allows us to use the Lemma \ref{lemma}. More precisely,
\[\Res(r_{n-1}, r_{n})= L_{n-1}^{\deg(r_{n})-\deg(v_{n}x^l r_{n-2}^m)}\Res(r_{n-1}, v_n x^l r_{n-2}^m).\]
Before we calculate the expression for \(R_n\), we need to find the value of \( \deg(r_n)-\deg(v_nx^lr_{n-2}^m)\) for \(n \geq d+2\). For \( n =d+2\) we have 
\begin{align*}
\gamma_A(d+2)&= \gamma(d+2) = \deg(r_{d+2})-\deg(v_n x^l r_{d}^m)\\
&= k +km+i_dm^2-(l+i_dm)=m(i_d(m-1)+k)+k-l,
\end{align*} 
and for \(n>d+2\) we obtain
\begin{align*}
\gamma_A(n)&=\gamma(n)=\deg(r_{n})- \deg (v_{n}x^l r_{n-2}^m)\\
&= k\sum_{s=0}^{n-d-1}m^s+i_dm^{n-d} - \left(l +m\left(k\sum_{s=0}^{n-d-3}m^s+i_dm^{n-d-2} \right)\right)\\
&=m^{n-d-1}(k+i_d(m-1))+k-l.
\end{align*}
Let \(e_A(n)=e(n)=\deg(r_{n}) \cdot\deg( r_{n-1})\), \(n\geq 1\). 
Then 
\begin{align*}
R_n&=\Res(r_{n}, r_{n-1}) \stackrel{(\ref{symmetry})}{=} (-1)^{e(n)}\Res(r_{n-1}, r_{n}) \\
&\stackrel{(\ref{division})}{=}  (-1)^{e(n)}L_{n-1}^{\gamma(n)}\Res(r_{n-1},v_{n}x^l r_{n-2}^m)\\
&\stackrel{(\ref{multiplicativity})}{=} (-1)^{e(n)}L_{n-1}^{\gamma(n)} \Res(r_{n-1}, v_n)\Res(r_{n-1},x)^l \Res(r_{n-1}, r_{n-2})^m\\
&\kern-0.6em\stackrel{(\ref{constant}),(\ref{zeros})}{=} (-1)^{e(n)}L_{n-1}^{\gamma(n)}v_n^{\deg r_{n-1}} (r_{n-1}(0))^l \Res(r_{n-1}, r_{n-2})^m\\
&=(-1)^{e(n)}L_{n-1}^{\gamma(n)}v_n^{\deg r_{n-1}} C_{n-1}^l R_{n-1}^m,
\end{align*}
where the numbers above the equals signs correspond to the numbers of the resultant properties in Section \ref{sect2}. With the convention that \(0^0=1\), the above equality is also valid for \(l=0\). By repeatedly applying the obtained relation we get
\[R_n = (-1)^{\sum_{s=d+1}^{n}m^{n-s}e(s)}\Res(r_{d},r_{d-1})^{m^{n-d}}\prod_{s=d+1}^{n} \left(L_{s-1}^{\gamma(s)}v_s^{\deg r_{s-1}}C_{s-1}^l \right)^{m^{n-s}}.\] 
\end{proof}
The above theorem can be translated to the case of the recurrence relation of order \(2\). More precisely, for \(d=1\) and an element \(A\) from the set 
\[\mathcal{A}= \{(i,j,k,l,m) \in \mathbb{N}^{5}: j\geq i, k\geq l, m\neq 0\}\]
we consider the sequence of the polynomials \((r_{A,n}(x))_{n\in\mathbb{N}}\) defined by the recurrence relation 
\[r_{A,n}(x)= \sum_{s=0}^m t_{s,n}(x)r_{A,n-1}^{m-s}(x)r_{A,n-2}^s(x)\]
with the initial conditions
\[r_{A,0}(x)= \sum_{s=0}^i p_sx^s, \quad r_{A,1}(x)=\sum_{s=0}^j q_s x^s,\]
where for every \(n \geq 2\)
\begin{align*}
t_{0,n}(x)&=a_{k,n}x^k+a_{k-1,n}x^{k-1}+\ldots+ a_{1,n}x+a_{0,n},\\
t_{m,n}(x)&=v_nx^l,\\
t_{s,n}(0)&=0\; \text{and} \; \deg t_{s,n} < \deg t_{0,n} \text{ for } s \in \{1,2,\ldots, m-1\}.
\end{align*}
Moreover, we assume that \(p_iq_ja_{k,n}\neq 0\) for every \(n \in \mathbb{N}_{\geq 2}\), and \(a_{k,2}p_j^m+v_2q_j^m\neq 0\) when \(i=j\) and \(k=l\). Then we have the following result:
\begin{cor}\label{order2}
Under the above assumptions, for \(n \in \mathbb{N}_{\geq 2}\) the resultant \(R_n\) of the polynomials \(r_{A,n}\), \(r_{A,n-1}\) is given by the formula:
\[R_n= (-1)^{\sum_{s=2}^n m^{n-s}e_A(s)}\prod_{s=2}^n \left(L_{s-1}^{\gamma_A(s)}v_s^{\deg r_{A,s-1}} C_{s-1}^l \right)^{m^{n-s}}R_1^{m^{n-1}},\]
where 
\begin{align*}
e_A(n)&=\deg (r_{A,n}) \cdot \deg (r_{A,n-1}),\\
\gamma_A(n)&=\deg (r_{A,n})-\deg (t_{m,n} r_{A,n-2}^m)\\
&= \begin{cases}
k-l +m(j-i) &\text{ for } n=2,\\
m^{n-2}(k+j(m-1))+k-l &\text{ for } n\geq 3,
\end{cases}\\
C_n &= \begin{cases}
1 &\text{ for } l=0,\\
q_0^{m^{n-1}}\prod_{s=2}^na_{0,n}^{m^{n-s}} & \text{ for } l>0,
\end{cases}\\
L_n &=
\begin{cases}
(a_{k,2}q_j^m +v_2p_j^m)^{m^{n-2}} \prod_{s=2}^n a_{k,s}^{m^{n-s}} &\text{\emph{for}} \; i=j \text{ and } k=l,\\
q_j^{m^{n-1}} \prod_{s=2}^n a_{k,s}^{m^{n-s}} &\text{ otherwise}.
\end{cases}
\end{align*}  
\end{cor}
From the statement of the Corollary \ref{order2} it is easier to see why our result is a generalization of the work \cite{ulas2021generalization} and the theorem attributed to Schur. For \(d=1\), \(m=1\) and arbitrary \(i,j,k,l \in \mathbb{N}\) we get the result presented by Ulas. When \(d=1\) and \(A=(0,1,1,0,1)\), the form of the considered recurrence relation is the same as in the proof of Schur's formula.

\bibliographystyle{acm}
\bibliography{bib}
\end{document}